\documentclass[a4paper,12pt]{amsart}
\usepackage{amssymb}
\input xypic
\newtheorem{thm}{Theorem}[section]
\newtheorem{lemma}[thm]{Lemma}
\newtheorem{cor}[thm]{Corollary}
\newtheorem{prop}[thm]{Proposition}
\theoremstyle{definition}
\newtheorem{rem}[thm]{Remark}

\def\GA{{\mathfrak{A}}}
\def\Aut{\operatorname{Aut}\nolimits}
\def\BC{{\mathbf{C}}}
\def\BF{{\mathbf{F}}}
\def\BG{{\mathbf{G}}}
\def\Gal{\operatorname{Gal}\nolimits}
\def\GL{\operatorname{GL}\nolimits}
\def\tH{{\tilde{H}}}
\def\Hom{\operatorname{Hom}\nolimits}
\def\BL{{\mathbf{L}}}
\def\Out{\operatorname{Out}\nolimits}
\def\PSL{\operatorname{PSL}\nolimits}
\def\PSU{\operatorname{PSU}\nolimits}
\def\BQ{{\mathbf{Q}}}
\def\GS{{\mathfrak{S}}}

\def\BT{{\mathbf{T}}}
\def\BZ{{\mathbf{Z}}}

\def\iso{\buildrel \sim\over\to}

\def\ie{{\em i.e.}}

\title{Automizers as extended reflection groups}
\author{Rapha\"el Rouquier}
\address{Mathematical Institute,
University of Oxford, 24-29 St Giles', Oxford, OX1 3LB, UK}
\email{rouquier@maths.ox.ac.uk}

\begin{document}
\maketitle

\section{Introduction}

Brou\'e, Malle and Michel have shown that the automizer
of an abelian Sylow $p$-subgroup $P$ in a finite simple Chevalley group is an
irreducible complex reflection
group (for $p$ not too small and different from the defining characteristic)
\cite{BrMaMi,BrMi}.

The aim is this note is to show that a suitable version of this property
holds for general finite groups.

\smallskip
We give a simple direct proof, building on
the Lehrer-Springer theory \cite{LeSp}, that
the property above holds for simply connected simple
algebraic groups $G$, provided $p$ is not a torsion prime
(Proposition \ref{pr:reflChevalley}): the automizer
$E=N_G(P)/C_G(P)$ is a reflection group on $\Omega_1(P)$, the largest
elementary abelian subgroup of $P$.

On the other hand, we show that the presence of $p$-torsion in the
Schur multiplier of a finite group $G$ prevents the subgroup of $E$
generated by reflections from being irreducible (Proposition
\ref{pr:notirreducible}).

\smallskip
This suggests considering covering groups of finite simple groups or
equivalently finite simple groups $G$ such that $H^2(G,\BF_p)=0$.
We also need to allow $p'$-automorphisms and we now
look for a description of the automizer
as an extension of an irreducible reflection group $W$ by
a subgroup of $N_{\GL(\Omega_1(P))}(W)/W$.

We actually need a slight generalization:
$\Omega_1(P)$ should be viewed in some cases as a vector space over
a larger finite field (for example in the case of $\PSL_2(\BF_{p^n})$)
and we need to allow field automorphisms.

As an example, the automizer of an $11$-Sylow subgroup in the
Monster is
the $2$-dimensional complex reflection group $G_{16}$.

\smallskip
I thank Richard Lyons and Geoff Robinson for useful discussions.

\section{Notation and definitions}
Let $p$ be a prime.
Given $P$ an abelian group, we denote by $\Omega_1(P)$ the
subgroup of $P$ of elements of order $1$ or $p$, \ie, the largest elementary
abelian $p$-subgroup of $P$. 

Let $V$ be a free module of finite rank over a commutative algebra $K$. A 
{\em reflection} is an element $s\in\GL_K(V)$ of finite order such that
$V/\ker(s-1)$ is a free $K$-module of rank $1$ (note that we do not
require $s^2=1$). A finite subgroup of
$\GL_K(V)$ is a {\em reflection group} if it is generated by reflections.

\section{Main result and remarks}

Let $p$ be a prime and $H$ a simple group such that the $p$-part
of the Schur multiplier of $H$ is trivial, \ie\ $H^2(H,\BF_p)=0$.
Assume $H$ has an abelian Sylow $p$-subgroup $P$.
Let $\tH\le\Aut(H)$ be a finite group containing $H$ and such that
$\tH/H$ is a Hall $p'$-subgroup of $\Out(H)$. Let $E=N_\tH(P)/C_\tH(P)$.

\begin{thm}
\label{th:main}
There is 
\begin{itemize}
\item a finite field $K$
\item an $\BF_p$-subspace $V$ of $\Omega_1(P)$ and
 an isomorphism of $\BF_p$-vector spaces
$K\otimes_{\BF_p}V\iso \Omega_1(P)$ endowing $\Omega_1(P)$ with a structure
of $K$-vector space
\item a subgroup $N$ of $\GL_K\bigl(\Omega_1(P)\bigr)$ and
\item a subgroup $\Gamma$ of $\Aut(K)$
\end{itemize}
such that $E=N\rtimes\Gamma$, as subgroups of $\Aut\bigl(\Omega_1(P)\bigr)$, and
such that the normal subgroup $W$ of $N$ generated by reflections
acts irreducibly on $\Omega_1(P)$.
\end{thm}

The theorem will be proven in \S \ref{se:proof}.

\begin{rem}
Gorenstein and Lyons have shown that $N_H(P)/C_H(P)$ acts irreducibly
on $\Omega_1(P)$ viewed as a vector space over $\BF_p$ and, as a consequence,
 $P$ is homocyclic \cite[(12.1)]{GoLy}.
Note nevertheless that the subgroup of $N_H(P)/C_H(P)$ generated by reflections
might not be irreducible
in its action on $\Omega_1(P)$: this happens for example in
the case $H=\GA_{2p}$, $p>3$.
\end{rem}

We can take $K=\BF_p$ in Theorem \ref{th:main}, except for
\begin{itemize}
\item $\PSL_2(p^n)$, $n>1$: $K=\BF_{p^n}$
\item $J_1$ and ${^2G}_2(q)$, $p=2$: $K=\BF_8$.
\end{itemize}
In those cases, $V=\BF_p$ and $P=\Omega_1(P)=K$.

\smallskip
Note that the theorem is trivial when $P$ is cyclic: one takes
$K=\BF_p$ and $N=E=W\subset \BF_p^\times$.

\medskip
Using the classification of finite simple groups, we deduce a statement
about finite groups with abelian Sylow $p$-subgroups.

\begin{cor}
\label{co:general}
Let $G$ be a finite group with an abelian Sylow $p$-subgroup $P$.
Let $H=O^{p'}(G/O_{p'}(G))$.

Assume the $p$-part of the Schur multiplier of $H$ is trivial. Then,
there is a finite group $X$ containing $H$ as a normal subgroup of
$p'$-index and
\begin{itemize}
\item a product $K$ of finite field extensions of $\BF_p$
\item an $\BF_p$-subspace $V$ of $\Omega_1(P)$ and
 an isomorphism of $\BF_p$-vector spaces
$K\otimes_{\BF_p}V\iso \Omega_1(P)$ endowing $\Omega_1(P)$ with a structure
of a free $K$-module
\item a subgroup $N$ of $\GL_K\bigl(\Omega_1(P)\bigr)$ and
\item a subgroup $\Gamma$ of $\Aut(K)$
\end{itemize}
such that $N_X(P)/C_X(P)=N\rtimes\Gamma$, as subgroups of
 $\Aut\bigl(\Omega_1(P)\bigr)$, and
such that denoting by $W$ the normal subgroup of $N$ generated by reflections,
we have $\Omega_1(P)^W=1$.
\end{cor}

\begin{proof}
The case where $H$ is simple is Theorem \ref{th:main}.
In general, the classification of finite simple groups shows that
there are finite simple groups $H_1,\ldots,H_r$ such that
$H=H_1\times\cdots\times H_r$ (cf. eg \cite[\S 5]{FoHa}).
Note that $O_p(H)=1$, \ie, there is no
non-trivial $p$-group as a direct factor of $H$, since $H^2(H,\BF_p)=0$.
Now, we take $X=X_1\times\cdots\times X_r$,
where the $X_i$ are associated with $H_i$. We put
$K=K_1\times\cdots\times K_r$, etc.
\end{proof}

\medskip
Following
\cite[Proof of (12.1)]{GoLy},
we give now the list of possible finite simple groups $H$ and primes $p$
such that Sylow $p$-subgroups of $H$ are abelian non-cyclic
and the $p$-part of the Schur multiplier of $H$ is trivial. In the first
case, instead of providing the group $H$, we provide a group $G$
such that $H\le G/O_{p'}(G)\le \Aut(H)$ and $p{\not|}\ [G/O_{p'}(G):H]$.

\begin{itemize}
\item $G=\BG^F$ where $\BG$ is a simply connected simple algebraic group
and $F$ is an endomorphism of $\BG$, a power of
which is a Frobenius endomorphism
defining a rational structure over a finite field with $q$ elements,
$p{\not|}q$ and $p$ is not a torsion prime for $\BG$
\item $H=\GA_n$ and $n<p^2$
\item $H=\PSL_2(p^n)$
\item $H={^2G}_2(q)$, $p=2$
\item $H$ is sporadic
\end{itemize}

Assume $K=\BF_p$. We have $V=\Omega_1(P)$ and $\Gamma=1$. Furthermore,
$N=E\subset N_{\GL(P)}(W)$. So, in this case, the theorem is equivalent
to the statement that $W$ acts irreducibly on $P$. As a consequence,
in order to show that the theorem holds, it is enough
to prove the statement with $\tH$
replaced by a group $G$ as above.

\begin{rem}
The finite simple groups with an abelian Sylow $p$-subgroup such that
the $p$-part of the Schur multiplier is non-trivial are the following
(cf \cite{Atl}):
\begin{itemize}
\item $H=M_{22},ON,\GA_6,\GA_7$ and $p=3$
\item $H=\PSL_2(q)$, $q\equiv 3,5 \pmod 8$ and $p=2$
\item $H=\PSL_3(q)$ and $3|q-1$ or $H=\PSU_3(q)$ and $3|q+1$ (here $p=3$)
\end{itemize}
Note that the automizer of a Sylow $3$-subgroup $P$ in $\Aut(ON)=ON.2$ does
not contain any reflection (when $P$ is viewed as a vector space over
$\BF_3$). That automizer is not a subgroup of $\GL_2(9).2$ (extension
by the Frobenius).
\end{rem}

Note that the presence of $p$-torsion in the Schur multiplier is an
obstruction to the irreducibility of
the subgroup of the automizer generated by reflections on
$\Omega_1(P)$, viewed as a vector space over $\BF_p$.

\begin{prop}
\label{pr:notirreducible}
Let $G$ be a finite group with an abelian Sylow $p$-subgroup $P$.
Let $E=N_G(P)/C_G(P)$ and let $W$ be the subgroup of $E$ generated by
reflections on $\Omega_1(P)$, viewed as an $\BF_p$-vector space.
Assume $p>2$.

If $H^2(G,\BF_p)\not=0$, then $\Omega_1(P)^W\not=0$.
\end{prop}

\begin{proof}
Let $V=\Omega_1(P)^*$. We have
$H^2(G,\BF_p)\simeq H^2(N_G(P),\BF_p)\simeq H^2(P,\BF_p)^E$.
On the other hand, we have an isomorphism of $\BF_p E$-modules
$H^2(P,\BF_p)\iso V\oplus\Lambda^2(V)$, so
$H^2(G,\BF_p)\simeq V^E\oplus \Lambda^2(V)^E\subset V^W\oplus \Lambda^2(V)^W$.
By Solomon's Theorem \cite{So}, we have $\Lambda^2(V)^W\simeq \Lambda^2(V^W)$.
The result follows.
\end{proof}

\begin{rem}
Let $W$ be a reflection group on a complex vector space $L$, with minimal field
of definition $K$.
The subgroup of the outer automorphism group of $W$ of elements fixing
the set of reflections 
has always a decomposition as a semi-direct
product $(N_{\GL(L)}(W)/W)\rtimes \Gal(K/\BQ)$
as shown by Marin and Michel \cite{MaMi}.
\end{rem}

\begin{rem}
It would be interesting to investigate if there is a version of Theorem
\ref{th:main} for non-principal blocks with abelian defect groups.

In a work in progress, we study automizers of maximal elementary abelian
$p$-subgroups in covering groups of simple groups.
\end{rem}

\section{Proof of Theorem \ref{th:main}}
\label{se:proof}
We run through the list of groups $H$ (or $G$) as described above.

\subsection{Chevalley groups}

Let $\BG$ be a connected and simply connected reductive
algebraic group over an algebraic closure $k$ of a finite field
and endowed with an endomorphism $F$, a power of which is a Frobenius
endomorphism.
Let $G=\BG^F$.
Assume $p$ is invertible in $k$ and $p$ is not a torsion prime for $\BG$.

\subsubsection{Abelian $p$-subgroups}
Since $p$ is not a torsion prime for $\BG$,
every abelian $p$-subgroup $Q$ of $G$ is contained in an $F$-stable
maximal torus $\BT$ of $\BG$
and $\BL=C_{\BG}(Q)$ is a Levi subgroup
(\cite[Corollary 5.10 and Theorem 5.8]{SpSt} and \cite[Proposition 2.1]{GeHi}).
Furthermore, $N_{\BG}(Q)=N_G(Q)C_{\BG}(Q)$ \cite[Corollary 5.10]{SpSt}, hence
$N_{\BG}(Q)/C_{\BG}(Q)=N_G(Q)/C_G(Q)$.

\smallskip
Let $W=N_{\BG}(\BT)/\BT$, $X=\Hom(\BT,\BG_m)$ and $Y=\Hom(\BG_m,\BT)$.
If $\BG$ is simple, then the action of $W$ on $\BC\otimes_\BZ X$ is
irreducible.

We have a canonical map
$N_W(Q)\to N_{\BG}(Q)/\BT$. Since
$\BL\subset N_{\BG}(Q)\subset N_{\BG}(\BL)$, we obtain an isomorphism
$$N_W(Q)/C_W(Q)\iso N_{\BG}(Q)/C_{\BG}(Q).$$

\smallskip
Given $L$ an abelian group, we denote by $L_{p^\infty}$ the subgroup of
$p$-elements of $L$.
Let $\mu=k^\times$.
We have an isomorphism
$$\BT_{p^\infty}\iso \Hom(X,\mu_{p^\infty}),\
t\mapsto (\chi\mapsto \chi(t)).$$
This provides an isomorphism
$$\BT_{p^\infty}\iso Y\otimes_{\BZ}\mu_{p^\infty}.$$
These isomorphisms are equivariant for the actions of $W$ and $F$.

\subsubsection{Abelian Sylow $p$-subgroups}
Assume now $P=Q$ is a abelian Sylow $p$-subgroup of $G$.
Let $V=Y\otimes_\BZ\BF_p$. We have $V^F\simeq\Omega_1(P)$.

\begin{prop}
\label{pr:reflChevalley}
The group $N_W(P)/C_W(P)$ is a reflection group on $\Omega_1(P)$.
If $\BG$ is simple, then this reflection group is irreducible.
\end{prop}

\begin{proof}
Note that $N_W(P)/C_W(P)$ is a $p'$-group, since $P$ is an abelian
Sylow $p$-subgroup of $G$ and $N_W(P)/C_W(P)\simeq N_G(P)/C_G(P)$. So,
the canonical map is an isomorphism
$$N_W(P)/C_W(P)\iso N_W(\Omega_1(P))/C_W(\Omega_1(P)).$$
The proposition follows now from the next lemma by Lehrer-Springer theory
\cite{LeSp} extended to positive characteristic \cite{Rou}.
\end{proof}

\begin{lemma}
We have
$\dim V^F\ge \dim V^{wF}$ for all $w\in W$.
\end{lemma}

\begin{proof}
Let $\dot{w}\in N_\BG(\BT)$.
 By Lang's Lemma, there is $x\in\BG$ such that
$\dot{w}=x^{-1}F(x)$. Given $t\in\BT$, we have $F(xtx^{-1})=x\dot{w}F(t)
\dot{w}^{-1}$. So, $x\BT x^{-1}$ is $F$-stable and
the isomorphism
$$\BT\iso x\BT x^{-1},\ t\mapsto xtx^{-1}$$
transfers the action of $wF$ on the left to the action of $F$ on the right.
So,
$$V^{wF}\simeq \bigl(Y(x\BT x^{-1})\otimes\BF_p\bigr)^F\simeq
\Omega_1\bigl((x\BT x^{-1})^F\bigr).$$
The rank of that elementary abelian $p$-subgroup of $G$ is at most the rank
of $P$ and we are done.
\end{proof}

\subsection{Alternating groups}
Let $G=\GS_n$, $n>7$. Put $n=pr+s$ with $0\le s\le p-1$ and
$r<p$. We have $P\simeq (\BZ/p)^r$. We put $K=\BF_p$, 
$N=W=\BF_p^\times \wr \GS_r$.

\begin{rem}
Note that when $n=5$ and $p=2$ or $n=6,7$ and $p=3$,
the $p$-part of the Schur multiplier is not trivial but the description
above is still valid. Note though that when
$n=6$ and $p=3$, then $G$ contains $\GS_6$ as a subgroup of index
$2$. We have $K=\BF_3$, $P=K^2$, $N=E$, $W$ is
a Weyl group of type $B_2$ and $[N:W]=2$.
\end{rem}

\subsection{$\PSL_2$}
Assume $H=\PSL_2(K)$ for a finite field $K$ of characteristic $p$. We have
$W=N=K^\times$ and $\Gamma=\Gal(K/\BF_p)$.

\subsection{${^2G}_2(q)$}
Assume $H={^2G}_2(q)$ and $p=2$. We have
$K=\BF_8$, $W=N=K^\times$ and $\Gamma=\Gal(K/\BF_2)$.

\subsection{Sporadic groups}
We refer to \cite{BrMaRou} for the diagrams for complex reflection groups.
For sporadic groups, we have $P=\Omega_1(P)$.

$$\begin{array}{|c|c|c|c|c|c|c|}
\cline{1-7}
\tH & K & \dim_K(P) & W & N/W & \Gamma & \text{diagram of W} \\
\cline{1-7}
J_1 & \BF_8 & 1 & \BF_8^\times & 1 & \Gal(\BF_8/\BF_2) &
{\small \xy (0,0) *++={7} *\frm{o} \endxy} \\
\cline{1-7}
M_{11},M_{23},HS.2 & \BF_3 & 2 & B_2 & 2 & 1 &
\xy
(0,0) *\cir<6pt>{}="0",
(13,0) *\cir<6pt>{}="1",
"0";"1" **\dir2{-},
\endxy
\\
\cline{1-7}
J_2.2,Suz.2 & \BF_5 & 2 & G_2 & 2 & 1 & 
\xy
(0,0) *\cir<6pt>{}="0",
(13,0) *\cir<6pt>{}="1",
"0";"1" **\dir3{-},
\endxy
\\
\cline{1-7}
He.2,Fi_{22}.2,Fi_{23},Fi_{24} & \BF_5 & 2 & G_8 & 1 & 1 &
{\small \xy (0,0) *++={4} *\frm{o} ; (13,0) *++={4} *\frm{o} **@{-}
   \endxy} \\
\cline{1-7}
Co_1 & \BF_7 & 2 & G_5 & 1 & 1 &
{\small \xy (0,0) *++={3} *\frm{o} ; (13,0) *++={3} *\frm{o} **@{=} \endxy} \\
\cline{1-7}
Th, BM & \BF_7 & 2 & G_5 & 2 & 1 & 
{\small \xy (0,0) *++={3} *\frm{o} ; (13,0) *++={3} *\frm{o} **@{=} \endxy} \\
\cline{1-7}
M & \BF_{11} & 2 & G_{16} & 1 & 1 &
{\small \xy (0,0) *++={5} *\frm{o} ; (13,0) *++={5} *\frm{o} **@{-} \endxy} \\
\cline{1-7}
\end{array}$$

\end{document}